\documentclass[reqno,12pt]{amsart}
\usepackage{amsmath,amstext, amsthm, amssymb,epsfig,color,tikz,mathrsfs}
\usetikzlibrary{positioning}
\numberwithin{equation}{section}
\newtheorem{thm}{Theorem}[section]
\newtheorem{lem}[thm]{Lemma}

\newtheorem{cor}[thm]{Corollary}

\newcommand{\be}{\begin{equation}}
\newcommand{\ee}{\end{equation}}
\newcommand{\ba}{\begin{array}}
\newcommand{\ea}{\end{array}}

\renewcommand{\em}{\it}

\newcommand{\al}{\alpha}

\newcommand{\bea}{\begin{eqnarray}}
\newcommand{\eea}{\end{eqnarray}}

\begin{document}
\title[Differential equations for Septic theta functions]{Differential equations for Septic theta functions}

\author{Tim Huber and Danny Lara}
\address{Department of Mathematics, University of Texas - Pan American, 1201 West University Avenue,  Edinburg, Texas 78539,  USA}

%\address{Department of Mathematics, Iowa State University, 396 Carver Hall, Ames, IA, USA}

%\thanks{The author would like to thank George Andrews, Bruce Berndt, and 
%  Ismail for their generous help and encouragement.}

%\keywords{Hadamard Products, $q$-Airy
%  Function, $q$-Bessel Function,
%  Rogers-Ramanujan Series, Generalized Stieltjes-Wigert Polynomials, Ramanujan's
%  Eisenstein series}


\begin{abstract}
We demonstrate that quotients of septic theta functions appearing in
S. Ramanujan's Notebooks and in F. Klein's work satisfy a new coupled system of nonlinear
differential equations with interesting symmetric form. This differential system bears a close
resemblance to an analogous system for quintic theta functions. The
proof extends a technique used by Ramanujan to prove the classical
differential system for normalized Eisenstein series on the full
modular group. In the course of our work, we show that Klein's quartic relation induces new symmetric
representations for low weight Eisenstein series in terms of weight
one modular forms of level seven.  
\end{abstract}

\maketitle

% \footnotetext[1]{Research partially supported by The UTPA Undergraduate Research Initiative.}

%\keywords{Ramanujan's Eisenstein series,
 % differential equations for Eisenstein series, parametric representations
 % for Eisenstein series, hypergeometric differential equation, hypergeometric
 % functions.}

%\subjclass{Primary 11F11; Secondary 33C05}.

%\vskip2pt

%\bigskip
%\vspace{-0.3in}
\section{Introduction}

Let $|q|< 1$, and 
\begin{align*}
 a(q) = - \sum_{n=-\infty}^{\infty} (-1)^{n} q^{(14n +5)^{2}/56}&, \qquad  b(q) = \sum_{n=-\infty}^{\infty} (-1)^{n} q^{(14n +3)^{2}/56}, \\ 
 c(q) = \sum_{n=-\infty}^{\infty} &(-1)^{n} q^{(14n +1)^{2}/56}.
% \qquad q = e^{2 \pi i \tau}, \quad  \re \tau > 0.
\end{align*}
These septic theta functions have an interesting provenance. They
first appeared  in the 1879  work of Felix Klein \cite{MR1722419} and
were studied independently by Ramanujan
\cite[p. 300]{ramnote}. With entirely different motivations, Klein and
Ramanujan derived identities
between these theta functions equivalent to Klein's eponymous quartic relation
\begin{align}
  \label{eq:55}
  a^{3}(q) b(q) + b^{3}(q) c(q) + c^{3}(q) a(q) = 0.
\end{align}
The septic theta functions $a(q)$, $b(q)$, and
$c(q)$ also appear in connection with Ramanujan's seventh order
mock theta functions \cite{selberg}. We augment the
work of Klein and Ramanujan by formulating a new coupled system of differential equations for these theta functions. Our approach is based on Ramanujan's famous proof of the coupled differential system for the normalized Eisenstein series on the full modular group \cite{Raman1}
{\allowdisplaybreaks \begin{equation} 
\label{rdiff1}  q \frac{dE_{2}}{dq} = \frac{E_{2}^{2} - E_{4}}{12}, \qquad  q \frac{dE_{4}}{dq} = \frac{E_{2}E_{4} - E_{6}}{3}, \qquad q \frac{dE_{6}}{dq} = \frac{E_{2}E_{6} - E_{4}^{2}}{2},
\end{equation}}
where the Eisenstein series $E_{k} = E_{k}(q)$ are defined by
\begin{align} \label{eis_full}
  E_{2k}(q) = 1 + \frac{2}{\zeta(1 - 2 k)} \sum_{n=1}^{\infty} \frac{n^{2k-1} q^{n}}{1 - q^{n}},
\end{align}
and where $\zeta$ is the analytic continuation of the Riemann $\zeta$-function.
 Ramanujan proved \eqref{rdiff1} by formulating two identities involving the classical Weierstrass zeta function 
\begin{align} \label{wzdef}
\zeta(\theta \mid q) &= \frac{1}{2} \cot \frac{\theta}{2} + \frac{\theta}{12} - 2 \theta \sum_{n=1}^{\infty} \frac{n q^{n}}{1 - q^{n}} + 2 \sum_{n=1}^{\infty} \frac{q^{n} \sin n \theta}{1 - q^{n}}.
\end{align}

We follow Ramanujan's lead to derive a new differential system from elementary properties of elliptic functions. Our work culminates in a curiously symmetric coupled differential system for the quotients 
%of septic theta functions 
\begin{align} \label{xy}
  x(q) = %\frac{q (q^{7};q^{7})_{\infty}^{2}(q^{2}; q^{7})_{\infty} (q^{5}; q^{7})_{\infty}}{(q^{3};q^{7})_{\infty}^{2} (q^{4}; q^{7})_{\infty}^{2}} 
%\frac{f(-q^{2}, -q^{5})}{f^{2}(-q^{3}, -q^{4})} 
 q^{7/8} (q^{7}; &q^{7})_{\infty}^{3}\frac{b(q)}{c^{2}(q)}, 
\qquad  y(q) =  
%q (q^{7};q^{7})_{\infty}^{3} \frac{f(-q, -q^{6})}{f^{2}(-q^{2}, -q^{5})} 
 -q^{7/8} (q^{7}; q^{7})_{\infty}^{3}\frac{a(q)}{b^{2}(q)}, \\   
&z(q) = 
%(q^{7};q^{7})_{\infty}^{3} \frac{f(-q^{3}, -q^{4})}{f^{2}(-q, -q^{6})} =  
q^{7/8} (q^{7}; q^{7})_{\infty}^{3}\frac{c(q)}{a^{2}(q)}, \label{z}
\end{align}
where here and throughout the paper, we employ the notation $$(a;q)_{n} =  \prod_{k=0}^{n-1} (1 - a q^{k}), \qquad (a;q)_{\infty} = \lim_{n \to \infty} \prod_{k=0}^{n-1} (1 - a q^{k}).$$
\begin{thm} \label{d_sept}  Let $\mathcal{P}(q) = E_{2}(q^{7})$, where $ E_{2}(q) = 1 - 24 \sum_{n=1}^{\infty} \bigl ( \sum_{d \mid n} d \bigr ) q^{n}$. Then 
%Let $x = x(q), y = y(q)$, $z = z(q)$, and $\mathcal{P}(q) = E_{2}(q^{7})$. Then
  \begin{align} \label{deqx}
    q\frac{d}{dq}x &= \frac{x}{12} \Bigl ( 5 y^{2} + 5z^{2} - 7x^{2}-
    20 yz -52 xy  + 7 \mathcal{P} \Bigr ), \\ q\frac{d}{dq}y
    &=\frac{y}{12} \Bigl ( 5z^{2 } + 5x^{2}  - 7 y^{2} + 20 xz -52yz + 7 \mathcal{P} \Bigr ), \label{deqy} \\ q\frac{d}{dq}z &= \frac{z}{12} \Bigl ( 5x^{2} + 5y^{2} - 7 z^{2} - 20 xy+ 52xz + 7 \mathcal{P}  \Bigr ), \label{deqz}
  \end{align}
  \begin{align}
     q \frac{d}{dq} \mathcal{P}(q) = &= \frac{7}{12} \Bigl (\mathcal{P}^2 -x^4+4 x^3 y+12 x y^3-y^4-12 x^3 z+4 y^3 z-4 x z^3+12 y z^3-z^4 \Bigr ). \nonumber
  \end{align}
\end{thm}
The form and symmetry present in this differential system is also
exhibited by a recently derived coupled system satisfied by
quintic theta functions \cite{huber_jac}, defined by
\begin{align*}
    A(q) &= q^{1/5} (q;q)_{\infty}^{-3/5} \sum_{n=-\infty}^{\infty} (-1)^{n} q^{(5n^{2} - 3n)/2}, \quad  B(q) = (q;q)_{\infty}^{-3/5} \sum_{n=-\infty}^{\infty} (-1)^{n} q^{(5n^{2} - n)/2}.
\end{align*}
\begin{thm} \label{d_quint} Let $\mathscr{P}(q) = E_{2}(q^{5})$, where $ E_{2}(q) = 1 - 24 \sum_{n=1}^{\infty} \bigl ( \sum_{d \mid n} d \bigr ) q^{n}$. Then 
  \begin{align}
    q \frac{d}{dq} A &= \frac{1}{60} A \Bigl (7 B^{10}-5 A^{10}-66 A^5 B^5+5 \mathscr{P} \Bigr ), \\ 
q \frac{d}{dq} B &= \frac{1}{60} B \Bigl (7 A^{10}-5 B^{10}+66 A^5 B^5+5 \mathscr{P} \Bigr ), 
  \end{align}
  \begin{align}
    q \frac{d}{dq} \mathscr{P} &= \frac{5}{12}\left( \mathscr{P}^{2}-B^{20}+ 12 B^{15}A^5 - 14 B^{10}A^{10}- 12 B^5 A^{15}- A^{20}\right). \label{dpp}
  \end{align}
\end{thm}
Our formulas for Eisenstein series in Theorem \ref{fina}
will demonstrate the equivalence of the fourth differential equation
of Theorem \ref{d_sept} and the first equation of Ramanujan's
differential system \eqref{rdiff1}. To accomplish this, we rely on septic parameterizations for
Eisenstein series equivalent to formulas appearing in Ramanujan's Lost
Notebook \cite{ramlost}. 

Each differential system appearing here is analogous to corresponding nonlinear coupled systems for modular forms of lower level. In particular, Theorems \ref{d_quint} and \ref{d_sept} are analogous to several coupled systems for the cubic theta functions \cite{cubic,maier-2008} 
%which are representable  \cite{MR1071759,MR1010408, MR1641658, MR1243610} in terms of Eisenstein series for $\Gamma_{0}(3)$
\begin{align}
\label{maier} \displaystyle q\frac{d}{dq} a &= \frac{a \mathscr{P} - b^{3}}{3}, \qquad \displaystyle q \frac{d}{dq} \mathscr{P} = \frac{\mathscr{P}^{2} - a b^{3}}{3}, \qquad \displaystyle q \frac{d}{dq} b^{3} = \mathscr{P} b^{3} - a^{2}b^{3},
%& q\frac{d}{dq}a = \frac{c^{3}-a\mathcal{P}}{3},\qquad q\frac{d}{dq}\mathcal{P}=\frac{ac^{3}-\mathcal{P}^{2}}{3},\qquad q\frac{d}{dq}c^{3}=c^{3}a^{2}-\mathcal{P}c^{3},  
\end{align}
where 
\begin{align}
a(q) &= \sum_{m,n =-\infty}^{\infty} q^{n^{2} + n m + m^{2}},  \quad
b(q) = \sum_{m,n =-\infty}^{\infty} \omega^{n-m} q^{n^{2} + n m + m^{2}}, \label{abq} \\ 
c(q) &= \sum_{m,n =-\infty}^{\infty}  q^{ \left ( n + \frac{1}{3} \right )^{2} + \left ( n + \frac{1}{3} \right ) \left ( m + \frac{1}{3} \right ) + \left ( m + \frac{1}{3} \right )^{2}}, \quad \omega = e^{2 \pi i/3}, \label{cq} \\ 
& \qquad \qquad \displaystyle \mathscr{P}(q) = 1 -6 \sum_{n=1}^{\infty} \frac{\cos (2 n \pi/3) n q^{n}}{1 - q^{n}}.
\end{align}
Similar coupled systems of differential equations for modular forms of
level $2$, $4$, and $6$ appear in \cite{maier-2008} and
\cite{huber,ramamani1}, respectively. Most of the coupled systems
discussed so far are subsumed (see \cite[\S 1]{huber_quint}) by a
more general system  for the parameters \cite{huber_proc} 
  \begin{align} \nonumber
    e_{\alpha}(q) = 1 + & 4\tan( \pi \alpha) \sum_{n=1}^{\infty} \frac{\sin( 2 n  \pi \alpha) q^{n}}{1 - q^{n}}, \quad P_{\alpha}(q) = 1 - 8 \sin^{2}( \pi \alpha) \sum_{n=1}^{\infty} \frac{\cos (2 n \pi \alpha) n q^{n}}{1 - q^{n}},  \\ & Q_{\alpha}(q) = 1- 8\tan( \pi \alpha) \sin^{2}( \pi \alpha) \sum_{n=1}^{\infty} \frac{\sin( 2 n  \pi \alpha) n^{2} q^{n}}{1 - q^{n}}. \label{ser}
  \end{align}

\begin{thm}  \label{mainthm} Let $e_{\alpha}(q), P_{\alpha}(q)$, and $Q_{\alpha}(q)$ be defined as in \eqref{ser}. Then for $\alpha \not \equiv 1/2$,
\allowdisplaybreaks{\begin{align}
\label{deqe}  q \frac{d}{dq}e_{\alpha} &= \frac{\csc^{2}(\pi \alpha )}{4} \left (  e_{\alpha} P_{\alpha} - Q_{\alpha} \right ), \\ 
\label{deqp} q \frac{d}{dq}P_{\alpha} &= \frac{\csc^{2}( \pi \alpha)}{4} P_{\alpha}^{2} - \frac{1}{2} \cot^{2}( \pi \alpha) e_{\alpha} Q_{\alpha} + \frac{1}{2}\cot( \pi \alpha) \cot(2\pi \alpha)e_{1 - 2 \alpha} Q_{\alpha}, \\
%\frac{\csc^{2}(\alpha \pi)}{4} \Bigl \{ P_{\alpha}^{2} - Q_{\alpha} \cot(\alpha) \left [ 2e_{\alpha} \cot(\alpha) + e_{1 - 2\alpha} \cot( 1 - 2 \alpha)\right ] \Bigr \}\\ 
\nonumber  q \frac{d}{dq}Q_{\alpha} &= \frac{1}{4} Q_{\alpha} P_{\alpha} \csc^{2}( \pi \alpha) + \frac{1}{2} P_{1 - 2 \alpha} Q_{\alpha} \csc^{2}(2\pi  \alpha) - \frac{1}{2} e_{1 - 2 \alpha}^{2} Q_{\alpha} \cot^{2}(2 \pi \alpha) \\ & \qquad + \frac{3}{2} e_{\alpha} e_{1 - 2 \alpha} Q_{\alpha} \cot( \pi \alpha) \cot(2 \pi \alpha) - e_{\alpha}^{2}Q_{\alpha} \cot^{2}( \pi \alpha). \label{deqq} 
%  P_{\alpha} Q_{\alpha} + e_{\alpha} Q_{\alpha} \cdot  \Bigl \{2 e_{\alpha} + e_{1 - 2\alpha} \cot( 1 - 2 \alpha)\Bigr \}
\end{align}}
\end{thm}
In Section 2, we formulate relevant elliptic function identities in terms of these parameters. In Section 3, we write the quotients $x(q)$,
$y(q)$, and $z(q)$ as linear combinations of $e_{1/7}(q)$,
$e_{2/7}(q)$, and $e_{3/7}(q)$. Section 4 culminates in a proof of
Theorem \ref{d_sept} and introduces new a set of symmetric parameterizations for
Eisenstein series of weight four and six in
terms of the septic parameters \eqref{xy}--\eqref{z}. These formulas
constitute a septic
reprisal of symmetric quintic
parameterizations for Eisenstein series from \cite{qeis}.

\section{Elliptic modular preliminaries}
The purpose of this section is to introduce results from
the theory of elliptic modular functions necessary for our further
work. A critical component in our proof of Theorem \ref{d_sept} is Lemma \ref{lem_expan}, where Lambert series representations
are derived for polynomials of degree two in the parameters
$e_{\alpha}(q)$. These rather unconventional parameters are
customarily expressed in terms of the logarithmic derivative of the Jacobi theta function
%\begin{align}
 % \frac{\theta_{1}'}{\theta_{1}} ( z\mid q) = \cot z + 4 \sum_{n=1}^{\infty} \frac{q^{n}}{1 - q^{n}} \sin 2 n z,
%%i - 2 i \sum_{n=1}^{\infty} \frac{q^{n} e^{2 i z}}{ 1 - q^{n} e^{2 i z}} + 2 i \sum_{n=0}^{\infty} \frac{q^{n} e%^{- 2 i z}}{1 - q^{n} e^{- 2 i z}},
%\end{align}
%where
\begin{align}
    \theta_{1}(z \mid q) = -iq^{1/8} \sum_{n=-\infty}^{\infty} (-1)^{n} q^{n(n+1)/2} e^{(2n+1)iz}
  \end{align}
given by the equivalent representations \cite[p. 489]{witwat}
\begin{align}
 \frac{\theta_{1}'}{\theta_{1}} ( z\mid q)  &= \cot z + 4
 \sum_{n=1}^{\infty} \frac{q^{n}}{1 - q^{n}} \sin 2 n z \\
&= i - 2 i \sum_{n=1}^{\infty} \frac{q^{n} e^{2 i z}}{ 1 - q^{n} e^{2 i z}} + 2
 i \sum_{n=0}^{\infty} \frac{q^{n} e^{- 2 i z}}{1 - q^{n} e^{- 2 i
     z}}. \label{another}
\end{align}
We will also require the familiar identity for the Jacobi
theta function \cite[p. 518]{witwat}
  \begin{align} \label{eq:13}
    \left ( \frac{\theta_{1}'}{\theta_{1}} \right )' (x \mid q) -
    \left ( \frac{\theta_{1}'}{\theta_{1}} \right )' (y \mid q) =
    \frac{\theta_{1}'(0 \mid q)^{2} \theta_{1}(x - y \mid q) \theta_{1}(x+y \mid q)}{\theta_{1}^{2}(x \mid q) \theta_{1}^{2}(y \mid q)}.
  \end{align}
%as well as a more exotic yet classical identity \cite{MR1189748,watson_german}
 % \begin{align}
  %  \label{eq:12}
   % \frac{\theta_{1}'}{\theta_{1}} & (x \mid q) +
    %  \frac{\theta_{1}'}{\theta_{1}}( y \mid q) +
    %  \frac{\theta_{1}'}{\theta_{1}}( z \mid q) -
     % \frac{\theta_{1}'}{\theta_{1}}( x + y + z \mid q) \\ 
%&= \theta_{1}'( 0 \mid q) \frac{\theta_{1}(x + y \mid q) \theta_{1}( y
 % + z \mid q) \theta_{1}( z + x \mid q)}{\theta_{1}( x \mid q)
 % \theta_{1}( y \mid q) \theta_{1}( z \mid q) \theta_{1}( x + y + z
 % \mid q)}. 
 % \end{align}
In order to relate the parameters $e_{\alpha}(q)$ to the quotients of theta
functions appearing in \eqref{xy}--\eqref{z}, we will logarithmically differentiate
the infinite product representations 
  \begin{align} \label{xyz_prod}
     x(q) &= \frac{q (q^{7};q^{7})_{\infty}^{2}(q^{2}; q^{7})_{\infty} (q^{5}; q^{7})_{\infty}}{(q^{3};q^{7})_{\infty}^{2} (q^{4}; q^{7})_{\infty}^{2}}, \quad 
  y(q) = \frac{q(q^{7}; q^{7})_{\infty}^{2}(q;q^{7})_{\infty}(q^{6}; q^{7})_{\infty}}{(q^{2};q^{7})_{\infty}^{2}(q^{5}; q^{7})_{\infty}^{2}}, \\
   & \qquad \quad \qquad \qquad z(q) = \frac{(q^{7}; q^{7})_{\infty}^{2}(q^{3};q^{7})_{\infty}(q^{4}; q^{7})_{\infty}}{(q;q^{7})_{\infty}^{2}(q^{6}; q^{7})_{\infty}^{2}}.
  \end{align}
 These product formulations are consequences of the Jacobi triple
 product formula \cite{witwat}
\begin{align} \label{jtp}
  \theta_{1}(z \mid q) = - i q^{1/8} e^{iz} (q;q)_{\infty} (q e^{2 i z}; q)_{\infty} (e^{- 2 i z}; q)_{\infty}.
\end{align}
By differentiating \eqref{jtp} at the origin, we obtain 
\begin{align} \label{yi1}
  \theta_{1}'(q) := \lim_{z \to 0} \frac{\theta_{1}(z \mid q)}{z} = 2 q^{1/8}(q;q)_{\infty}^{3}.  
\end{align}
We may also apply \eqref{jtp} to derive the subsequently useful product representations
\begin{align}
  \label{eq:88}
  \theta_{1}( \pi \tau \mid q^{7}) &= i q^{3/8} (q ; q^{7})_{\infty}
  (q^{6}; q^{7})_{\infty} (q^{7}; q^{7})_{\infty}, \\ 
  \theta_{1}( 2\pi \tau \mid q^{7}) &= i q^{-1/8} (q^{2} ; q^{7})_{\infty}
  (q^{5}; q^{7})_{\infty} (q^{7}; q^{7})_{\infty}, \label{eq:89} \\ 
  \theta_{1}( 3\pi \tau \mid q^{7}) &= i q^{-3/8} (q^{3} ; q^{7}_{\infty}
  (q^{4}; q^{7})_{\infty} (q^{7}; q^{7})_{\infty}. \label{eq:90}
\end{align}

Our proof of Theorem \ref{d_sept} will employ a number of classical elliptic function identities for the Weierstrass $\zeta$-function defined by \eqref{wzdef} and the Weierstrass $\wp$-function determined by $(d/dz) \zeta( z \mid q) = - \wp(z \mid q).$
%When the nome $q$ is understood, we adopt the convention $\wp(z \mid q) = \wp(z)$. 
%$$\zeta( 2\theta \mid q) = \frac{e_{\theta}(q) \cot \theta}{2}. +\frac{E_{2}(q)}{12}$$
%Equations \eqref{ramtrig1} and \eqref{ramtrig2} represent the Fourier expansions for the elliptic function identities 
In particular, we will make use of an identity prominent in Ramanujan's proof of the differential system for Eisenstein series \eqref{rdiff1}. Ramanujan used elementary trigonometric identities to prove \cite[p.~135]{hardy} (cf. \cite{spiritraman})
% \cite[p.~436]{witwat}
\begin{align} \label{ell1} 
  \left ( \zeta(z \mid q) - \frac{z E_{2}(q)}{12} \right )^{2} &= \wp(z \mid q)  - \frac{1}{6} + 4\sum_{n=1}^{\infty} \frac{q^{2n}\cos(n z)}{(1 - q^{2n})^{2}}.
%\\ 
% (\wp'(z \mid q))^{2} = 4 \wp^{3}(&z \mid q) - \frac{1}{12}E_{4}(q) \wp(z) - \frac{1}{216}E_{6}(q). \label{ell2}
\end{align}
%Ramanujan expressed \eqref{ell1} and \eqref{ell2} in the form given by \eqref{ramtrig2}--\eqref{ramtrig1} and equated coefficients in the Maclaurin expansions to prove the differential system \eqref{rdiff1}. 
Identity \eqref{ell1} played a key role in Ramanujan's proof \cite{Raman1} of
\eqref{rdiff1} and facilitates the main results of the present paper by inducing Lambert series expansion for
$e_{\alpha}^{2}(q)$. Corresponding Lambert expansions for $e_{\alpha}(q) e_{1 - 2 \alpha}(q)$ will depend on the Frobenius-Stickelberger pseudo-addition formula \cite{fstick}, \cite[p.~459]{witwat} for the Weierstrass $\zeta$-function
\begin{align} \label{stickel}
\{ \zeta( a ) + \zeta(b) + \zeta(c) \}^{2} = \zeta'(a) + \zeta'(b) + \zeta'(c), \qquad a + b + c = 0.
%2 \zeta(2 u) - 4 \zeta(u) &= \frac{\wp''(u)}{\wp'(u)},
\end{align}
%Equation \eqref{hg} may be obtained from the classical iden%tity \cite[p.~55]{chandra}
%\begin{align}
%\frac{\wp'(u)}{\wp(u) - \wp(v)} &= \zeta( u - v) + \zeta( u% + v) - \zeta(u). \label{adf}  
%\end{align}
The next lemma translates \eqref{ell1}--\eqref{stickel} into forms involving the series from \eqref{ser}.
\begin{lem} \label{lemp} Let $e_{\alpha}$ and $P_{\alpha}$ be defined by \eqref{ser}, and let $E_{2}(q)$ denote the normalized Eisenstein series of weight $2$. Then \label{lem_expan}
  \begin{align}
%    \left (\frac{1}{4}\cot(\pi \alpha) + \sum_{n=1}^{\infty} \frac{q^{n}\sin(2 \pi \alpha n)}{1 - q^{n}} \right )^{2} &= \frac{1}{16} \cot^{2}(\pi \alpha) + \sum_{n=1}^{\infty} \frac{q^{n} \cos(2 \pi \alpha n)}{(1 - q^{n})^{2}} + \frac{1}{2} \sum_{n=1}^{\infty} \frac{n q^{n}}{1 - q^{n}} - \frac{1}{2} \sum_{n=1}^{\infty} \frac{nq^{n} \cos(2 \pi \alpha n)}{1 - q^{n}} \\
\label{one} e_{\alpha}^{2}(q) &= 1 + \frac{16}{\cot^{2}(\pi\alpha)} \sum_{n=1}^{\infty}  \frac{q^{n} \cos(2 \pi \alpha n)}{(1 - q^{n})^{2}} + \frac{8}{\cot^{2}(\pi \alpha)} \sum_{n=1}^{\infty} \frac{n q^{n}}{1 - q^{n}} \\ & \qquad \qquad  \qquad - \frac{8}{\cot^{2}(\pi \alpha)} \sum_{n=1}^{\infty} \frac{q^{n} n\cos(2 n \pi \alpha)}{1 - q^{n}}, \nonumber \\
& \label{two} \Bigl (  \cot(\pi ( 1 - 2 \al))  e_{1 - 2 \alpha}(q) + 2 \cot ( \pi \alpha)e_{\alpha}(q) \Bigr )^{2} \\ & \qquad \qquad \qquad = \csc^{2} ((1 - 2 \alpha) \pi) P_{1 - 2 \alpha}(q) + 2 \csc^{2} ( \pi \al ) P_{\alpha}(q) - E_{2}(q). \nonumber
\end{align}
\end{lem}
\begin{proof}
Recast \eqref{ell1} in the form 
\begin{align} \label{ramtrig1}
&  \left ( \frac{1}{4} \cot \theta + \sum_{n=1}^{\infty} \frac{q^{n} \sin ( 2 n \theta)}{ 1 - q^{n}} \right )^{2} \\ &\qquad = \left ( \frac{1}{4} \cot \theta \right )^{2} + \sum_{n=1}^{\infty} \frac{q^{n} \cos (2 n \theta)}{(1 - q^{n})^{2}} + \frac{1}{2} \sum_{n=1}^{\infty} \frac{n q^{n}}{1 - q^{n}} ( 1 - \cos ( 2 n \theta)). \nonumber
\end{align} 
Equation \eqref{one} follows from setting $\theta = 2\pi \alpha$ in
\eqref{ramtrig1}. Identity \eqref{stickel} takes the form
\begin{align} \label{ha}
  & \left ( \frac{1}{2} \cot \theta - \cot \frac{\theta}{2} + \sum_{n=1}^{\infty} \frac{q^{n} (2 \sin(2 n \theta) - 4 \sin(n \theta))}{1 - q^{n}} \right )^{2} \\ & \qquad = \frac{1}{4} \csc^{2} \theta + \frac{1}{2} \csc^{2} \frac{\theta}{2} - \sum_{n=1}^{\infty} \frac{n q^{n}(2 \cos(2 n \theta) + 4 \cos(n \theta))}{1 - q^{n}}. \nonumber
\end{align}
Equation \eqref{two} may be obtained by setting $\theta = 2\pi \alpha$ in \eqref{ha}.
\end{proof}

\begin{lem} \label{ba} Let $e_{\alpha}(q)$ be defined by
  \eqref{ser}. Then 
  \begin{align} \label{db1}
e_{\alpha}^{2}(q) &=  1 + \sum_{n=1}^{\infty} \frac{\delta_{\alpha}(n)n q^{n}}{1 - q^{n}}  + \sum_{n=1}^{\infty} \frac{\lambda_{\alpha}(n) q^{n}}{(1 - q^{n})^{2}}, \\ 
  e_{\alpha}(q) e_{1 - 2 \alpha}(q) &= 1 + \sum_{n=1}^{\infty} \frac{\kappa_{\alpha}(n)n q^{n}}{1 - q^{n}}  + \sum_{n=1}^{\infty} \frac{\mu_{\alpha}(n) q^{n}}{(1 - q^{n})^{2}} , \label{db2}
  \end{align}
where 
\begin{align*}
\delta_{\alpha}(q)  = 16 \tan ^2(\pi  \alpha) &\sin ^2(\pi  n \alpha
), \qquad \lambda_{\alpha}(n) = 16 \tan ^2(\pi  \alpha) \cos (2 \pi
n \alpha ), \\ 
 \kappa_{\alpha}(n) &= 8 \tan (\pi  \alpha) \tan (2 \pi  \alpha ) \sin
 ^2(\pi  n \alpha ), \\ 
% -8\tan \Bigl ( (1 - 2 \alpha) \pi \Bigr ) \tan (\alpha \pi) \sin^{2}(an\pi) \\ 
 \mu_{\alpha}(n)  = 4\tan &(\pi  \alpha) \tan (2 \pi  \alpha) \Bigl (4
 \cos  (2 \pi  n \alpha )+\cos(4 \pi n \alpha ) \Bigr ).
%\frac{1}{4} \tan \Bigl (\pi  (1-2 a) \Bigr ) \tan (\pi  a) \Big (-64 \cos (2 \pi  a n)-16 \cos \big (2 \pi  (1-2 a) n \big ) \Big )
\end{align*}
\end{lem}

\begin{proof}
Equation \eqref{db1} follows from \eqref{one} and elementary trigonometric identities.
%\begin{align*}
% 1 - \cos x  = 2 \sin^2 x.
%\end{align*}
To prove \eqref{db2}, expand the left side of \eqref{two} and subtract the squared terms from both sides to obtain
\begin{align} \label{ie}
 e_{1 - 2 \alpha}(q) e_{\alpha}(q) &= \frac{\tan(\pi \alpha) \tan (2 \pi \alpha)}{4} \Bigl ( E_{2}(q) +  \cot^{2}((1 - 2 \alpha)\pi) e_{1-2 \alpha}^{2}(q)  \\ \nonumber & \qquad   +  4 \cot^{2}(\pi \alpha) e_{\alpha}^{2}(q)  -\csc^{2} ((1 - 2 \alpha) \pi) P_{1 - 2 \alpha}(q) - 2 \csc^{2} ( \pi \al ) P_{\alpha}(q)  \Bigr ).
\end{align}
Next, apply \eqref{db1} on the right side of \eqref{ie} to the terms $$\cot^{2}((1 - 2 \alpha)\pi) e_{1-2 \alpha}^{2}(q)\qquad \hbox{and} \qquad 4 \cot^{2}(\pi \alpha) e_{\alpha}^{2}(q).$$ Thus, \eqref{ie} may be expressed in the form
\begin{align} \label{sl1}
 e_{1 - 2 \alpha}(q) e_{\alpha}(q) &= \sum _{n=1}^{\infty} \frac{8
   \tan (2 \pi \alpha ) \tan (\pi  \alpha) \sin ^2(\pi  n \alpha ) n
   q^{n}}{1-q^n} \\ &+\sum _{n=1}^{\infty} \frac{4\tan (2 \pi \alpha )
   \tan (\pi  \alpha) \bigl (4 \cos (2 \pi   n\alpha ) + \cos (4 \pi n
   \alpha ) \bigr ) q^n}{ \left(1-q^n\right)^2}. \nonumber
\end{align}
Comparing the right side of \eqref{sl1} with \eqref{db2}, we arrive at the claimed identity.
\end{proof}

\begin{lem} \label{mip}
  For any sequence $\{a_{n}\}_{n=1}^{\infty}$ periodic modulo seven, such that the the series are absolutely convergent, we have 
  \begin{align} \label{mi}
    \sum_{n=1}^{\infty} \frac{a_{n} q^{n}}{(1 - q^{n})^{2}} = \sum_{n=1}^{\infty} \frac{n }{1 - q^{7n}} \sum_{m=1}^{7} a_{m} q^{nm}.
  \end{align}
\end{lem}
\begin{proof}
  To prove \eqref{mi}, express the sum on the left as the
  derivative of a geometric series and invert the order of summation
  to yield
  \begin{align}
\label{eq:144}    \sum_{n=1}^{\infty} \frac{a_{n} q^{n}}{(1 - q^{n})^{2}}
&=\sum_{n=1}^{\infty} a_{n} q^{n} \frac{d}{dq^{n}} \left ( \frac{1}{1
    - q^{n}} \right ) 
%\\ \nonumber &= \sum_{n=1}^{\infty} a_{n} q^{n} \sum_{k=1}^{\infty} k
q^{n(k-1)} 
%\\ \nonumber &= \sum_{n=1}^{\infty} a_{n} \sum_{k=1}^{\infty} k q^{nk}
%\\ \nonumber &=  \sum_{m=1}^{7} \sum_{n=0}^{\infty} a_{7n+m}
%\sum_{k=1}^{\infty} k q^{(7n+m)k} 
 =  \sum_{m=1}^{7} a_{m} \sum_{k=1}^{\infty} kq^{mk}
\sum_{n=0}^{\infty} q^{7nk}. 
%\\ \nonumber &=  \sum_{m=1}^{7} a_{m} \sum_{k=1}^{\infty}
%\frac{kq^{mk}}{1 - q^{7k}} 
%\\ &= \sum_{k=1}^{\infty} \frac{k}{1 - q^{7k}} \sum_{m=1}^{7} a_{m}q^{mk}. \nonumber
  \end{align}
By expanding the innermost sum of \eqref{eq:144} as a geometric series,
we obtain \eqref{mi}.
\end{proof}

\section{Elliptic interpolation of septic theta functions}
 %To connect the septic theta functions with the Weierstrass zeta function, we use the corresponding Jacobi triple product representations. The relevant identities between theta functions and linear combinations of the Weierstrass zeta functions are proven by expressing the logarithmic derivatives of both the product and Lambert series in terms of identical linear combinations of Eisenstein series of weight two.  We list below the Jacobi triple product representations for each of the theta quotients from \eqref{xy}--\eqref{z}. 

%\begin{lem}
%  \begin{align}
%     x(q) &= \frac{q (q^{7};q^{7})_{\infty}^{2}(q^{2}; q^{7})_{\infty} (q^{5}; q^{7})_{\infty}}{(q^{3};q^{7})_{\infty}^{2} %(q^{4}; q^{7})_{\infty}^{2}}, \\ 
%  y(q) &= \frac{q(q^{7}; q^{7})_{\infty}^{2}(q;q^{7})_{\infty}(q^{6}; q^{7})_{\infty}}{(q^{2};q^{7})_{\infty}^{2}(q^{5}; q^%{7})_{\infty}^{2}}, \\
%z(q) &= \frac{(q^{7}; q^{7})_{\infty}^{2}(q^{3};q^{7})_{\infty}(q^{4}; q^{7})_{\infty}}{(q;q^{7})_{\infty}^{2}(q^{6}; q^{7}%)_{\infty}^{2}}.
%  \end{align}
%\end{lem}
We now apply the results of the Section 2 to study the functions $x(q), y(q)$, and $z(q)$. 
Our goal in the next Lemma is to obtain representations $x(q)$,
$y(q)$, and $z(q)$ as linear combinations of logarithmic derivatives of theta functions denoted $e_{1/7}(q), e_{2/7}(q), e_{3/7}(q)$. 
\begin{lem} \label{ba1}
  \begin{align}
    x(q) &= \alpha_{1} e_{1/7}(q) + \alpha_{2} e_{2/7}(q) + \alpha_{3}e_{3/7}(q), \label{x_fund}\\ 
    y(q) &= \beta_{1} e_{1/7}(q) + \beta_{2} e_{2/7}(q) + \beta_{3}e_{3/7}(q), \\ 
    z(q) &= \gamma_{1} e_{1/7}(q) + \gamma_{2} e_{2/7}(q) + \gamma_{2}e_{3/7}(q),  
  \end{align}
where 
\begin{align}
  \alpha_{1} = \frac{1}{14} \Bigl(1-3 \cos &\bigl(\frac{3 \pi }{14}\bigr) \csc \bigl(\frac{\pi }{7}\bigr)\Bigr),\quad
%\frac{-19+7 \sin \left(\frac{\pi }{14}\right)-18 \sin \left(\frac{3 \pi }{14}\right)+33 \cos \left(\frac{\pi }{7}\right)}{14 \left(-4+\sin \left(\frac{\pi }{14}\right)-3 \sin \left(\frac{3 \pi }{14}\right)+6 \cos \left(\frac{\pi }{7}\right)\right)}, \\ 
\alpha_{2} = \frac{1}{14} \Bigl(1+6 \sin \bigl(\frac{3 \pi }{14}\bigr)\Bigr), \\
%\frac{13-24 \sin \left(\frac{\pi }{14}\right)+22 \sin \left(\frac{3 \pi }{14}\right)-25 \cos \left(\frac{\pi }{7}\right)}{14 \left(-4+\sin \left(\frac{\pi }{14}\right)-3 \sin \left(\frac{3 \pi }{14}\right)+6 \cos \left(\frac{\pi }{7}\right)\right)}, \\ 
\alpha_{3} &= \frac{1}{14} \Bigl(1-6 \sin \bigl(\frac{ \pi }{14}\bigr)\Bigr),
%\frac{27-25 \sin \left(\frac{\pi }{14}\right)+38 \sin \left(\frac{3 \pi }{14}\right)-50 \cos \left(\frac{\pi }{7}\right)}{14 \left(-4+\sin \left(\frac{\pi }{14}\right)-3 \sin \left(\frac{3 \pi }{14}\right)+6 \cos \left(\frac{\pi }{7}\right)\right)},
\end{align}
\begin{align}
  \beta_{1} &= \frac{1}{56} \left(4-\csc \left(\frac{\pi }{14}\right) \left(2+\csc \left(\frac{3 \pi }{14}\right)\right)\right), \\ 
%\frac{-6+\sin \left(\frac{\pi }{14}\right)-4 \sin \left(\frac{3 \pi }{14}\right)+10 \cos \left(\frac{\pi }{7}\right)}{14 \left(-4+\sin \left(\frac{\pi }{14}\right)-3 \sin \left(\frac{3 \pi }{14}\right)+6 \cos \left(\frac{\pi }{7}\right)\right)}, \\ 
\beta_{2} &= \frac{1}{14} \left(1+4 \sin \left(\frac{3 \pi }{14}\right)-2 \cos \left(\frac{\pi }{7}\right)\right), \\ 
%\frac{\sin \left(\frac{\pi }{14}\right)+\sin \left(\frac{3 \pi }{14}\right)}{2 \left(3+\sin \left(\frac{\pi }{14}\right)-\sin \left(\frac{3 \pi }{14}\right)+\cos \left(\frac{\pi }{7}\right)\right)}, \\ 
\beta_{3} &= \frac{1}{28} \left(2-4 \sin \left(\frac{\pi }{14}\right)+\csc \left(\frac{3 \pi }{14}\right)\right),
%\frac{13-6 \sin \left(\frac{\pi }{14}\right)+14 \sin \left(\frac{3 \pi }{14}\right)-23 \cos \left(\frac{\pi }{7}\right)}{14 \left(-4+\sin \left(\frac{\pi }{14}\right)-3 \sin \left(\frac{3 \pi }{14}\right)+6 \cos \left(\frac{\pi }{7}\right)\right)},
\end{align}
\begin{align}
  \gamma_{1} &= \frac{1}{28} \left(4+\csc \left(\frac{\pi }{14}\right) \left(3+\csc \left(\frac{3 \pi }{14}\right)\right)\right), \\ 
%\frac{-7+4 \sin \left(\frac{\pi }{14}\right)-8 \sin \left(\frac{3 \pi }{14}\right)+9 \cos \left(\frac{\pi }{7}\right)}{14 \left(-4+\sin \left(\frac{\pi }{14}\right)-3 \sin \left(\frac{3 \pi }{14}\right)+6 \cos \left(\frac{\pi }{7}\right)\right)}, \\ 
\gamma_{2} &= \frac{1}{7} \left(1-5 \sin \left(\frac{3 \pi }{14}\right)+3 \cos \left(\frac{\pi }{7}\right)\right), \\ 
%\frac{-6+5 \sin \left(\frac{\pi }{14}\right)+9 \sin \left(\frac{3 \pi }{14}\right)-2 \cos \left(\frac{\pi }{7}\right)}{14 \left(-2-2 \sin \left(\frac{\pi }{14}\right)+\sin \left(\frac{3 \pi }{14}\right)+\cos \left(\frac{\pi }{7}\right)\right)}, \\
\gamma_{3} &= \frac{1}{28} \left(4+8 \sin \left(\frac{\pi }{14}\right)-3 \csc \left(\frac{3 \pi }{14}\right)\right).
%\frac{-130+255 \sin \left(\frac{\pi }{14}\right)-257 \sin \left(\frac{3 \pi }{14}\right)+254 \cos \left(\frac{\pi }{7}\right)}{14 \left(-35+63 \sin \left(\frac{\pi }{14}\right)-67 \sin \left(\frac{3 \pi }{14}\right)+59 \cos \left(\frac{\pi }{7}\right)\right)}.
\end{align}
\end{lem}

\begin{proof}
%  The claimed formulas may be derived from Lambert series
 % representations for the functions $x(q)$, $y(q)$, and $z(q)$. These
 % Lambert series expansions are consequences of theta function
 % identities derived by Z-G. Liu \cite{MR1767652}. Liu constructs
  %three elliptic functions and uses the residue theorem to derive
  %a triple of identities equivalent to 
%constructed three elliptic functions by employing the quasiperiodicity of the Jacobi theta function 
%%  \begin{align*}
%%    \theta_{1}(z \mid q) = -iq^{1/8} \sum_{n=-\infty}^{\inf%ty} (-1)^{n} q^{n(n-1)/2} e^{(2n+1)iz}
%%  \end{align*}
%and the fact that the sum of the residues of each elliptic function
%on a period parallelogram equals zero. By summing the residues Liu
%derived 
In order to obtain relevant Lambert series expansions for
the series from \eqref{xy}--\eqref{z}, we employ three theta
function identities
derived by Z.-G. Liu  \cite[pp. 67-68]{MR1767652} 
\begin{align} \label{yu1}
  q^{-1} \theta_{1}'(q^{7}) \frac{\theta_{1}(2 \pi \tau \mid q^{7})}{\theta_{1}^{2}(3 \pi \tau \mid q^{7})} &= - 2i+ \frac{\theta_{1}'}{\theta_{1}} (\pi \tau \mid q^{7}) -  \frac{\theta_{1}'}{\theta_{1}} (2\pi \tau \mid q^{7}) -2 \frac{\theta_{1}'}{\theta_{1}} (3\pi \tau \mid q^{7}), \\ 
 q^{-1/2} \theta_{1}'(q^{7}) \frac{\theta_{1}( \pi \tau \mid q^{7})}{\theta_{1}^{2}(2 \pi \tau \mid q^{7})} &= \frac{\theta_{1}'}{\theta_{1}} (\pi \tau \mid q^{7}) - 2\frac{\theta_{1}'}{\theta_{1}} (2\pi \tau \mid q^{7}) + \frac{\theta_{1}'}{\theta_{1}} (3\pi \tau \mid q^{7}), \\ 
 q^{1/2} \theta_{1}'(q^{7}) \frac{\theta_{1}(3 \pi \tau \mid q^{7})}{\theta_{1}^{2}( \pi \tau \mid q^{7})} &= 2i+ \frac{\theta_{1}'}{\theta_{1}} (\pi \tau \mid q^{7}) +  \frac{\theta_{1}'}{\theta_{1}} (2\pi \tau \mid q^{7}) + \frac{\theta_{1}'}{\theta_{1}} (3\pi \tau \mid q^{7}). \label{yu3}
\end{align}
%where
%\begin{align} \label{yi1}
 % \theta_{1}'(q) := \lim_{z \to 0} \frac{\theta_{1}(z \mid q)}{z} = 2 q^{1/8}(q;q)_{\infty}^{3}.  
%\end{align}
%Identity \eqref{yi1} is a consequence of the Jacobi triple product expansion \cite[p. 470]{witwat}
%\begin{align} \label{jtp}
 % \theta_{1}(z \mid q) = - i q^{1/8} e^{iz} (q;q)_{\infty} (q e^{2 i z}; q)_{\infty} (e^{- 2 i z}; q)_{\infty}
%\end{align}
%The right sides of \eqref{yu1}--\eqref{yu3} may be expressed in terms of Lambert series by logarithmically% differentiating \eqref{jtp}
%\begin{align}
% \frac{\theta_{1}'}{\theta_{1}} ( z\mid q) = i - 2 i \sum_{n=1}^{\infty} \frac{q^{n} e^{2 i z}}{ 1 - q^{n} e^{2 i z}} + 2 i \sum_{n=0}^{\infty} \frac{q^{n} e^{- 2 i z}}{1 - q^{n} e^{- 2 i z}}.
%\end{align}
%The left sides of \eqref{yu1}--\eqref{yu3} may be formulated in terms of infinite products by way of \eqref{jtp} and \eqref{yi1}. Some simplification allows us to recast \eqref{yu1}--\eqref{yu3} in the form 
These may be reformulated through the use of \eqref{another}, \eqref{yi1}, and
\eqref{eq:88}-\eqref{eq:90} as 
  \begin{align} \label{gh1}
    x(q) = \sum_{n=1}^{\infty} \frac{a_{n}q^{n}}{1 - q^{n}}, \qquad y(q) = \sum_{n=1}^{\infty} \frac{b_{n}q^{n}}{1 - q^{n}}, \qquad z(q) = 1 + \sum_{n=1}^{\infty} \frac{c_{n}q^{n}}{1 - q^{n}},
  \end{align}
where $a_{n}, b_{n}, c_{n}$ are periodic sequences modulo seven defined by 
\begin{align*}
  \{a_{n}\}_{n=0}^{7} = \{0, 1, -1, &-2, 2, 1, -1\}, \quad  \{b_{n}\}_{n=0}^{7} = \{0, 1, -2, 1, -1, 2, -1\}, \\  & \{c_{n}\}_{n=0}^{7} = \{0, 2, 1, 1, -1, -1, -2\}.
\end{align*}
The periodic odd sequence $a_{n}$ has the standard discrete Fourier representation \cite{zygmund}
\begin{align}
f(n) =  \sum_{m=1}^{3}\ell_{m}\sin \left ( \frac{2 \pi m x}{7} \right ), \quad \hbox{where} \quad  \ell_{m} = \frac{2}{7} \sum_{k=0}^{6} a_{k} \sin \left ( \frac{2 \pi m k}{7}\right ).
\end{align}
Thus, from \eqref{x_fund} and \eqref{gh1}, and referring to the definition of $e_{\alpha}(q)$ from \eqref{ser}, we see that the constants $\alpha_{j}$ appearing in \eqref{x_fund} are determined by $\alpha_{j} = \ell_{j}/(4 \tan( j \pi/7))$. The simplified form for each of these constants is given in Lemma \ref{ba1}. Each triple of constants $\beta_{j}$ and $\gamma_{j}$ may be similarly constructed as corresponding multiples of the finite Fourier coefficients for the sequences $b_{n}$ and $c_{n}$, respectively.
\end{proof}

To derive parameterizations for Eisenstein series necessary to prove
Theorem \ref{d_sept}, we
require a parameterization for the Hecke Eisenstein series of weight one  twisted by
the septic Jacobi symbol. This is an immediate consequence of the
formulas on line 
\eqref{gh1}. 
\begin{lem} \label{eiswt1} Let $\left (
        \frac{n}{7} \right )$ denote the Jacobi symbol modulo
  seven. Then 
  \begin{align}
    \label{eq:11}
      1+ 2 \sum_{n=1}^{\infty} \left (
        \frac{n}{7} \right ) \frac{q^{n}}{1 - q^{n}} = x(q) - y(q) + z(q). 
  \end{align}
\end{lem}

Our derivation of further parameterizations for Eisenstein
series depend fundamentally on Klein's quartic identity \eqref{eq:55}. We therefore give an elementary proof
of \eqref{eq:55}. From \eqref{xy}-\eqref{z}, we may rephrase equation
\eqref{eq:55} in terms of the quadratic \eqref{eq:10}.
\begin{lem}
  \begin{align}
    \label{eq:10}
    x(q) y(q) - x(q) z(q) + y(q) z(q) = 0. 
  \end{align}
\end{lem}
\begin{proof}

Replace $q$ by $q^{7}$ in \eqref{eq:13}, and make the respective
substitutions 
\begin{align}
  \label{eq:133}
  (x, y) = (\pi \tau, 2 \pi \tau), \quad (\pi \tau, 3 \pi \tau), \quad
  (2 \pi \tau, 3 \pi \tau)
\end{align}
in the resultant identities to derive, from \eqref{another},
\eqref{yi1}, and \eqref{eq:88}-\eqref{eq:90}, 
\begin{align} \label{kg} 
  D_{1}(q) - D_{2}(q) &= y(q) z(q),  \quad D_{1}(q) - D_{3}(q) = x(q)
  z(q), \\   & D_{2}(q) - D_{3}(q) = x(q) y(q), \label{kg1}
\end{align}
where $D_{1}(q)$, $D_{2}(q)$, and $D_{3}(q)$ take the form
\begin{align} \label{rn}
  D_{1}(q) = \sum_{n=1}^{\infty} \frac{n q^{n}}{1 - q^{7n}} +&
  \sum_{n=1}^{\infty} \frac{nq^{6n}}{1 - q^{7n}},\ \ D_{2}(q) =  \sum_{n=1}^{\infty} \frac{n q^{2n}}{1 - q^{7n}} + \sum_{n=1}^{\infty} \frac{nq^{5n}}{1 - q^{7n}}, \\ 
 &D_{3}(q) =  \sum_{n=1}^{\infty} \frac{n q^{3n}}{1 - q^{7n}} + \sum_{n=1}^{\infty} \frac{nq^{4n}}{1 - q^{7n}}. \label{rn1}
\end{align}
Identity \eqref{eq:10} follows immediately from \eqref{kg}--\eqref{kg1}.
\end{proof}

\section{A proof of the septic system}
We now present a proof of Theorem \ref{d_sept} through a sequence of
elementary lemmas. We first show that the logarithmic derivatives $x(q)$, $y(q)$,
and $z(q)$ coincide with the relevant quadratics on the right side of
\eqref{deqx}--\eqref{deqz}. To prove the final
equation of Theorem \ref{d_sept}, we derive a parameterization for
$E_{4}(q^{7})$ in terms of $x(q)$, $y(q)$, and $z(q)$. 
\begin{lem} \label{agp} Let $x = x(q), y = y(q)$, and $z =z(q)$ be defined by
  \eqref{xy}\eqref{z}. Then 
  \begin{align}
    \label{eq:14}
    q \frac{d}{dq} \log x & = \frac{5 y^{2} + 5z^{2} - 7x^{2}- 20 yz
      -52 xy  + 7 \mathcal{P}}{12} =  1 +
    \sum_{n=1}^{\infty} f(n) \frac{n q^{n}}{1 - q^{n}},  \\ 
    q \frac{d}{dq} \log y  &=\frac{5x^{2} + 5z^{2} - 7 y^{2} + 20
      xz -52yz + 7 \mathcal{P}}{12} =  1 +
    \sum_{n=1}^{\infty} g(n) \frac{n q^{n}}{1 - q^{n}}, \label{deqy} \\ q\frac{d}{dq}
    \log z &= \frac{5x^{2} + 5y^{2} - 7 z^{2} - 20 xy+ 52xz + 7
      \mathcal{P} }{12}=  1 +
    \sum_{n=1}^{\infty} h(n) \frac{n q^{n}}{1 - q^{n}}. 
  \end{align}
where $f,g$ and $h$ are periodic arithmetic functions modulo seven defined
by 
\begin{align}
  \label{eq:15}
  \{f(n)\}_{n=0}^{6} = \{-2,0,-1,&2,2,-1,0\}, \qquad \{g(n)\}_{n=0}^{6} =
  \{-2,-1,2,0,0,2,-1\}, \\ &\{h(n)\}_{n=0}^{6} = \{-2,2,0,-1,-1,0,2\}.
\end{align}
\end{lem}

\begin{proof} The claimed equality between the extreme sides of \eqref{eq:14}
  follows immediately from the Jacobi triple product representation
  \eqref{xyz_prod} for $x(q)$. 
%  \begin{align}
%    q \frac{d}{dq} \log x(q) = 1 + \sum_{n=1}^{\infty} f(n) \frac{n q^{n}}{1 - q^{n}},
%  \end{align}
%where
%\begin{align}
 % f(n) =
 % \begin{cases}
  %  -2, & \hbox{if}\ n \equiv 0 \pmod{7},  \\ 
  %  0, & \hbox{if}\ n \equiv 1,6 \pmod{7},  \\ 
   % 2, & \hbox{if}\ n \equiv 3, 4 \pmod{7},  \\ 
  %  -1, & \hbox{if}\ n \equiv 2, 5 \pmod{7}.  \\ 
  %\end{cases}
%\end{align}
%Next, we need to show
%\begin{align*}
 % 1 + \sum_{n=1}^{\infty} f(n) \frac{n q^{n}}{1 - q^{n}} =  a \cdot e_{1/7}^2+b \cdot  e_{2/7}^2+c \cdot  e_{3/7}^2+d \cdot % e_{2/7} e_{1/7}+f \cdot e_{3/7} e_{1/7}+h \cdot P(q^7),
%\end{align*}
%In order to show \eqref{deqx}, we shall prove 
%\begin{align} \label{comp}
%  \frac{5 y^{2} + 5z^{2} - 7x^{2}- 20 yz -52 xy  + 7 \mathcal{P}}{12} =  1 + \sum_{n=1}^{\infty} f(n) \frac{n q^{n}}{1 - q^{n}}.
%\end{align}
%In light of \eqref{comp}, we conclude that the left and right sides
%of \eqref{deqx} differ by a constant multiple. 
To prove the rightmost equality of \eqref{eq:14}, we begin by applying Lemma
\ref{ba1} to write the middle expression of \eqref{eq:14} as a polynomial of
degree two in $e_{1/7}(q)$, $e_{2/7}(q)$, and $e_{3/7}(q)$. With the
constants $\alpha_{i}, \beta_{i}, \gamma_{i}$, $i=1,2,3$, defined as in Lemma
\ref{ba1}, we find
\begin{align} \label{cap}
  5 y^{2} + 5z^{2}-52 &xy  - 7x^{2} - 20 yz \\ &=
  (5 \beta_1^2+5
  \gamma_1^2-52 \alpha_1 \beta_1-7 \alpha_1^2-20 \beta_1 \gamma_1 ) e_{1/7}^{2} \nonumber \\ &+ (5 \beta_2^2+5 \gamma_2^2-52 \alpha_2 \beta_2-7
  \alpha_2^2-20 \beta_2 \gamma_2)e_{2/7}^{2}
  \nonumber \\ &+ (5 \beta_3^2+5 \gamma_3^2-52 \alpha_3 \beta_3-7 \alpha_3^2-20 \beta_3
  \gamma_3)e_{3/7}^{2} \nonumber \\ +(10  \beta_1 \beta_2+10 \gamma_1 &
  \gamma_2-52
  \alpha_2\beta_1-52 \alpha_1 \beta_2 -14 \alpha_1 \alpha_2-20
  \beta_2 \gamma_1-20 \beta_1 \gamma_2)e_{1/7}e_{2/7} \nonumber  \\ + (10 \beta_2 \beta_3+10 \gamma_2 &\gamma_3-52 \alpha_3 \beta_2-52 \alpha_2
  \beta_3 -14 \alpha_2 \alpha_3-20 \beta_3 \gamma_2-20 \beta_2
  \gamma_3)e_{2/7}e_{3/7} \\
  \nonumber +(10 \beta_1 \beta_3+10 \gamma_1 &
  \gamma_3-52 \alpha_3
  \beta_1-52 \alpha_1 \beta_3 -14 \alpha_1 \alpha_3-20 \beta_3
  \gamma_1-20 \beta_1 \gamma_3)e_{1/7}e_{3/7} 
  \nonumber \\ :=  \Phi_{1} e_{1/7}^{2} &+\Phi_{2} e_{2/7}^{2}+
  \Phi_{3}e_{3/7}^{2} +\Phi_{4} e_{1/7}e_{2/7} + \Phi_{5}
  e_{1/7}e_{3/7} + \Phi_{6} e_{2/7}e_{3/7} 
%\\  &=1 + \sum_{k=1}^{3} \frac{\Phi_{k}}{12} \left ( \sum_{n=1}^{\infty} 
 % \frac{\delta_{k/7}(n)n q^{n}}{1 - q^{n}}  + \sum_{n=1}^{\infty}
 % \frac{\lambda_{k/7}(n) q^{n}}{(1 - q^{n})^{2}} \right ) \\ & \quad
%\quad   + \sum_{k=4}^{6} \frac{\Phi_{k}}{12} \left ( \sum_{n=1}^{\infty} 
%  \frac{\kappa_{k/7}(n)n q^{n}}{1 - q^{n}}  + \sum_{n=1}^{\infty}
%  \nonumber 
 % \frac{\mu_{k/7}(n) q^{n}}{(1 - q^{n})^{2}} \right ) 
\\ &= \sum_{k=1}^{6} \Phi_{k} + \sum_{n=1}^{\infty} 
  \frac{ a_{n} n q^{n}}{1 - q^{n}} + \sum_{n=1}^{\infty}  \frac{b_{n}
    q^{n}}{(1 - q^{n})^{2}}, \label{eq:fd}
\end{align}
where $a_{n}$ and $b_{n}$ are sequences periodic modulo seven defined
through Lemma \ref{ba} by 
\begin{align}
  \label{eq:16}
  a_{n} = \sum_{k=1}^{3} \Phi_{k}\delta_{k/7}(n) + \sum_{r=4}^{6}
  \Phi_{r}\kappa_{(k-3)/7}(n), \quad b_{n} =  \sum_{k=1}^{3} \Phi_{k} \lambda_{k/7}(n) + \sum_{r=4}^{6}
  \Phi_{r} \mu_{(k-3)/7}(n).
\end{align}
%where $a, b, c, d, f, h $ are constants in the file 12\_4\_12. This
%can be accomplished by using Lemma \ref{ba} and collecting the
%coefficients of the summands $\frac{n q^{n}}{1 - q^{7n}}$ and showing
%that the coefficients match $f(n)$ for values of $n$ in each residue
%class modulo seven. To finish the proof, we will apply Lemma
%\ref{ba1} to express the series $e_{1/7}(a), e_{2/7}(q), e_{3/7}(q)$
%arising on the right side of the last equation in terms of $x(q),
%y(q)$, and $z(q)$. 
By employing exact precision arithmetic in Mathematica
9.0 yields
\begin{align}
  \label{eq:155}
\{a_{n}\}_{n=0}^{6} = \{0, 37, 25, 61, & 61, 25, 37 \}, \quad \{b_{n}\}_{n=0}^{6} = \{
222, -37, -37, -37, -37, -37, -37\}, \nonumber  \\ 
& \Phi_{1} + \Phi_{2} + \Phi_{3} + \Phi_{4} + \Phi_{5} + \Phi_{6} = 5.
\end{align}
%Apply Lemma \ref{ba} to each term on the right side of \eqref{cap} to derive
%\begin{align} \nonumber 
%  \frac{5 y^{2} + 5z^{2} - 7x^{2} - 20 yz -52 xy  + 7 \mathcal{P}}{12}
%  &=1 + \sum_{k=1}^{3} \frac{\Phi_{k}}{12} \left ( \sum_{n=1}^{\infty} 
%  \frac{\delta_{k/7}(n)n q^{n}}{1 - q^{n}}  + \sum_{n=1}^{\infty}
%  \frac{\lambda_{k/7}(n) q^{n}}{(1 - q^{n})^{2}} \right ) \\ & \quad
%\quad   + \sum_{k=4}^{6} \frac{\Phi_{k}}{12} \left ( \sum_{n=1}^{\infty} 
%  \frac{\kappa_{k/7}(n)n q^{n}}{1 - q^{n}}  + \sum_{n=1}^{\infty}
%  \nonumber 
%  \frac{\mu_{k/7}(n) q^{n}}{(1 - q^{n})^{2}} \right ) \\ &= 1 + \sum_{n=1}^{\infty} 
%  \frac{ a_{n} n q^{n}}{1 - q^{n}} + \sum_{n=1}^{\infty}  \frac{b_{n} q^{n}}{(1 - q^{n})^{2}},
%\end{align}
%where $a_{n}$ and $b_{n}$ are periodic modulo seven
%defined by 
%\begin{align}
%  \label{eq:16}
%  a_{n} = \sum_{k=1}^{3} \frac{\Phi_{k} }{12}\delta_{k/7}(n) + \sum_{r=4}^{6}
%  \frac{\Phi_{r}}{12} \kappa_{k/7}(n), \qquad b_{n} =  \sum_{k=1}^{3} \Phi_{k} \lambda_{k/7}(n) + \sum_{r=4}^{6}
%  \Phi_{r} \mu_{k/7}(n).
%\end{align}
After applying Lemma \ref{mip}, the rightmost series of
\eqref{eq:16} takes the form
\begin{align}
  \label{eq:177}
  \sum_{n=1}^{\infty} \frac{b_{n} q^{n}}{(1 - q^{n})^{2}} 
%&=
 % \sum_{n=1}^{\infty} \frac{n}{1 - q^{7n}} \sum_{m=1}^{7} b_{m} q^{nm}
 = \sum_{n=1}^{\infty} \frac{n}{1 - q^{7n}} \sum_{m=1}^{7} (-37)
  q^{mn} + 259 \sum_{n=1}^{\infty} \frac{n q^{7n}}{ 1 - q^{7n}} =
  - 37 \sum_{\substack{n=1 \\ 7 \nmid n}}^{\infty} \frac{n q^{n}}{1 - q^{n}}.
\end{align}
Therefore, by applying the calculations on lines \eqref{cap}--\eqref{eq:155}, we may write
\begin{align}
  \label{eq:189}
   \frac{5 y^{2} + 5z^{2} - 7x^{2}- 20 yz
      -52 xy  + 7 \mathcal{P}}{12} =  1 + \sum_{\substack{n=1 \\ 7
        \nmid n}}^{\infty} \frac{(a_{n} -37)n q^{n}}{12(1 - q^{n})} -
    2\sum_{n=1}^{\infty} \frac{7 n q^{7n}}{1 - q^{7n}}. %\\  &=  1 +
%    \sum_{n=1}^{\infty} f(n) \frac{n q^{n}}{1 - q^{n}},
\end{align}
The final equality of \eqref{eq:189} demonstrates the truth of the
first claim of Lemma \ref{agp}. Proofs of the latter two
identities of the lemma are similar. We omit the details.
\end{proof}
%   \begin{align}
%    \label{eq:14}
%    q \frac{d}{dq}  x & = C_{1} \frac{5 y^{2} + 5z^{2} - 7x^{2}- 20 yz
%      -52 xy  + 7 \mathcal{P}}{12},  \\ 
%    q \frac{d}{dq} y  &=\frac{5x^{2} + 5z^{2} - 7 y^{2} + 20
%      xz -52yz + 7 \mathcal{P}}{12}, \label{deqy} \\ q\frac{d}{dq} z
%    &= C_{3} \frac{5x^{2} + 5y^{2} - 7 z^{2} - 20 xy+ 52xz + 7
 %     \mathcal{P} }{12}.
%  \end{align}
%\end{cor}

Our proof of Lemma \ref{agp} addresses the first three equations \eqref{deqx}--\eqref{deqz} of Theorem \ref{d_sept}. It remains for us to prove the last equation of Theorem
\ref{d_sept}. In the next sequence of lemmas we will show that this equation is a reparameterization
of Ramanujan's original differential equation for weight two
Eisenstein series given by \eqref{rdiff1}
\begin{align}
  \label{eq:9}
 q \frac{dE_{2}}{dq} = \frac{E_{2}^{2} - E_{4}}{12},  
\end{align}
To exhibit the equivalence of \eqref{eq:9} and the last equation of
Theorem \ref{d_sept}, it suffices to
derive a corresponding parameterization for the Eisenstein series of
weight $4$ and argument $q^{7}$; namely, 
\begin{align} \label{eq:4}
  E_{4}(q^{7}) &= x^4-4 x^3 y+12 x^3 z-12 x y^3+4 x z^3+y^4-4 y^3 z-12 y z^3+z^4.
%x^4+4 x^3 y-4 x y^3+y^4+4 x^3 z+8 x y z+8 x^2 y
%    z-8 x y^2 z\\ & \qquad +4 y^3 z -4 x z^3 -4 y z^3+z^4.  \label{eq:4}
\end{align}
The validity of Equation \eqref{eq:4}  will be addressed following our proof
of Theorem \ref{fina}. We construct relevant parameterizations for Eisenstein series in terms of septic parameters from Klein's relation \eqref{eq:10} as well as a representation appearing in the next lemma for the
Hauptmodul on $\Gamma_{0}(7)$ as a rational function of $x, y,z$. 

\begin{lem} \label{ring1} Let $x = x(q), y = y(q)$, and $z =z(q)$ be defined by
  \eqref{xy}\eqref{z}. Then 
  \begin{align}
    \label{eq:1}
 \frac{(q;q)_{\infty}^{4}}{q(q^{7}; q^{7})_{\infty}^{4}} = \frac{z^{2} - x z-y^2-6 y z}{yz}.
%\frac{yz}{z^{2} - x z-y^2-6 y z}.
% q \frac{(q^{7}; q^{7})_{\infty}^{4}}{(q;q)_{\infty}^{4}} = \frac{y
 %     z}{z^{2} - x z-y^2-6 y z}.
  \end{align}
\end{lem}

\begin{proof}
  From \cite[p. 88]{elkies} (cf. \cite[p. 838]{ lachaud}) we have
  \begin{align}
    \label{eq:3}
    j_{7} :=  \frac{(q;q)_{\infty}^{4}}{q(q^{7}; q^{7})_{\infty}^{4}} = \frac{ab^{5} + b c^{5} + ca^{5} - 5 a^{2} b^{2} c^{2}}{a^{2}  b^{2} c^{2}},
  \end{align}
where $a, b$ and $c$ are defined by
\eqref{abq}--\eqref{cq}. Klein's quartic relation \eqref{eq:55}
%\begin{align}
 % \label{eq:5}
 % a^{3} b + b^{3} c + c^{3} a = 0
%\end{align}
implies
\begin{align}
  \label{eq:6}
  ab^{5}+ b c^{5} + c a^{5} - 5 a^{2} b^{2} c^{2} = ab^{5} + ca^{5}
  + 6 b c^{5} + \frac{5 b^{4} c^{3}}{a}. 
\end{align} 
Hence, by \eqref{eq:3}, and \eqref{eq:6},
\begin{align}
  \label{eq:7}
  \frac{(q;q)_{\infty}^{4}}{q(q^{7}; q^{7})_{\infty}^{4}} &= \frac{a^{2}b^{5} + ca^{6}
  + 6 ab c^{5} + 5 b^{4} c^{3}}{a^{3}  b^{2} c^{2}} 
%\\ & = \frac { \left ( \frac{c}{a^{2}} \right )^{2} - \left
 % ( - \frac{a}{b^{2}} \right)^{2} - 6 \left ( -\frac{a}{b^{2}} \right
%) \left ( \frac{c}{a^{2}} \right ) - \left ( \frac{b}{c^{2}} \right )
%\left (  \frac{c}{a^{2}} \right )}{\left
 %   ( \frac{-a}{b^{2}} \right )
%\left (  \frac{c}{a^{2}} \right ) } \\ & 
= \frac{z^{2} - x z-y^2-6 y z}{yz},
\end{align}
where last equality of \eqref{eq:7} follows from \eqref{xy}--\eqref{z}.
\end{proof}

\begin{lem} \label{ring2}  Let $x = x(q), y = y(q)$, and $z =z(q)$ be defined by
  \eqref{xy}\eqref{z}. Then 
  \begin{align} \label{eq:hm}   q^{2} \frac{(q^{7}; q^{7})^{7}}{(q;q)_{\infty}} &=
   xyz, \qquad 
    % = x^2 (z-y)  \\
%\label{ss2}  
 q (q;q)_{\infty}^{3} (q^{7}; q^{7})_{\infty}^{3} = x \left(z^{2} -y^2 + 6xy -7 x z\right), \\
   \label{ss3} \frac{(q;q)_{\infty}^{7}}{(q^{7}; q^{7})_{\infty}} &= x^3-32 x^2 y+13 x y^2-y^3+45 x^2 z-13 x z^2+z^3.
  \end{align}
\end{lem}

\begin{proof}
  The leftmost equation on line \eqref{eq:hm} follows from the Jacobi triple product
  representations \eqref{xyz_prod} for $x, y,$ and $z$. To derive, the second equation of \eqref{eq:hm},
  multiply \eqref{eq:1} by the left equation of \eqref{eq:hm} to derive
  \begin{align}\label{wl}
     q (q;q)_{\infty}^{3} (q^{7}; q^{7})_{\infty}^{3} = x \left(z^{2} -y^2 + 6xy -7 x z\right) - 6x(xy - xz + y z).
  \end{align}
Klein's quartic
  relation \eqref{eq:10} may be applied to \eqref{wl} to arrive at the rightmost equation of \eqref{eq:hm}. Equation \eqref{ss3}
  may be derived by multiplying the second equation of \eqref{eq:hm} by equation \eqref{eq:1} and similarly applying Klein's quartic
  relation.
\end{proof}
We now employ formulas for Eisenstein series equivalent to those
appearing in Ramanujan's Lost Notebook \cite[p. 53]{ramlost}  (see also \cite{MR1600120,MR2511669,liusom}) to construct relevant
representations for Eisenstein series.  The following representations
were formulated by S. Cooper and P. C. Toh \cite[p. 176]{MR2511669}
from Ramanujan's septic representations for Eisenstein series.
\begin{lem} \label{coopt} Let $$\sigma
  =   1+ 2 \sum_{n=1}^{\infty} \left (
        \frac{n}{7} \right ) \frac{q^{n}}{1 - q^{n}} , \qquad Z = \frac{(q;q)_{\infty}^{7}}{(q^{7};
    q^{7})_{\infty}}, \qquad X =  \frac{q(q^{7}; q^{7})_{\infty}^{4}}{(q;q)_{\infty}^{4}}.$$
Then
\begin{align}
  \label{eq:2}
E_{4}(q) &= Z \sigma ( 1 + 245 X + 2401 X^{2}), \\ \label{eq7}
E_{4}(q^{7}) &= Z \sigma ( 1 + 5X  + X^{2}), \\ 
E_{6}(q) & = Z^{2} ( 1 - 490 X - 21609 X^{2} - 235298 X^{3} - 823543
X^{4}), \\ 
E_{6}(q^{7}) & - Z^{2}( 1 + 14 X + 63 X^{2} + 70 X^{3} - 7 X^{4}).
\end{align}  
\end{lem}

The parameterizations from Lemma \ref{coopt} may be transcribed in equivalent form as polynomials in $x(q)$, $y(q)$, and $z(q)$.  These beautiful
formulas comprise the last ingredient needed for our proof of Theorem
\ref{d_sept}. Their intentional symmetry is one of the infinitely many
equivalent septic formulations made possible by Klein's
relation \eqref{eq:10}.
\begin{thm} \label{fina}
\begin{align}
 \nonumber    E_{4}(q) &= x^4-116 x^3 y+116 x y^3+y^4-116 x^3 z+848 x y z^{2}+848 x^2 y z \\ & \qquad -848 x y^2 z-116 y^3 z+116 x z^3+116 y z^3+z^4, \\
%x^4-964 x^3 y-732 x y^3+y^4+732 x^3 z-964 y^3 z+964 x z^3-732 y z^3+z^4. \\ 
 \nonumber    E_{4}(q^{7}) &= x^4+4 x^3 y-4 x y^3+y^4+4 x^3 z+8 x y z+8 x^2 y
    z-8 x y^2 z\\ & \qquad +4 y^3 z -4 x z^3 -4 y z^3+z^4,  \label{sei}
\end{align}
\begin{align}  
    E_{6}(q) &= x^6+258 x^5 y-5904 x^4 y^2-5904 x^2 y^4-258 x
    y^5+y^6+258 x^5 z \\ & \qquad +7310 x^3 y^2 z +7310 x^2 y^3 z+258
    y^5 z-5904 x^4 z^2+7310 x^3 y z^2 \nonumber \\ & \qquad -8751 x^2
    y^2 z^2-7310 x y^3 z^2-5904 y^4 z^2-7310 x^2 y z^3 \nonumber \\ &\qquad -7310 x y^2 z^3 -5904 x^2 z^4-5904 y^2 z^4-258 x z^5-258 y z^5+z^6, \\
  E_{6}(q^{7}) &= x^6+6 x^5 y+18 x^4 y^2+18 x^2 y^4-6 x y^5+y^6+6 x^5
  z+2 x^3 y^2 z \nonumber \\ &\qquad +2 x^2 y^3 z +6 y^5 z+18 x^4
  z^2+2 x^3 y z^2-57 x^2 y^2 z^2-2 x y^3 z^2 \nonumber \\ &\qquad +18 y^4 z^2-2
  x^2 y z^3-2 x y^2 z^3+18 x^2 z^4+18 y^2 z^4 -6 x z^5 \nonumber \\ & \qquad -6 y z^5+z^6.
  \end{align}
\end{thm}

\begin{proof}
  To prove \eqref{sei}, apply \eqref{eq7} and the formulas from
  Lemmas \ref{ring1}-\ref{ring2} to derive 
  \begin{align} \label{frw}
     &Z \sigma ( 1 + 5X  + X^{2}) \\ & - (x^4+4 x^3 y+4 x^3 z+8 x^2 y
     z-4 x y^3-8 x y^2 z+8 x y z^2-4 x z^3+y^4+4 y^3 z-4 y z^3+z^4)
     \nonumber \\
 &= (x y - x z + y z) \frac{-42 x^4 z^2-79 x^3 y^2 z-258 x^3 y
   z^2+52 x^3 z^3-37 x^2 y^4}{\left(z x - z^{2}+y^2+6 y z\right)^2}
 \nonumber \\ 
&+  (x y - x z + y z) \frac{-184 x^2 y^3 z +3 x^2 y^2 z^2-6 x^2
  yz^3+30 x^2 z^4+45 x y^5}{\left(z x - z^{2}+y^2+6 y z\right)^2}
\nonumber 
\\ &+ (x y - x z + y z)  \frac{+292 x y^4 z+90 x y^3 z^2-449 x y^2
  z^3+334 x y z^4-48 x z^5}{\left(z x - z^{2}+y^2+6 y z\right)^2}
\nonumber 
\\ &+ (x y - x z + y z)  \frac{-10 y^6-84 y^5 z-133 y^4 z^2+96 y^3
  z^3+121 y^2 z^4-70 y z^5+8 z^6 }{\left(z x - z^{2}+y^2+6 y
    z\right)^2}. \nonumber
\end{align}
Klein's relation \eqref{eq:10} implies that the difference of the
expression on line \eqref{frw} and the line immediately following is
zero. Since the other relations of Theorem \ref{fina} are similarly derived, we omit the
details. 
\end{proof}
To derive the last equation of Theorem \ref{d_sept} and complete the
proof of Theorem \ref{d_sept}, note that by \eqref{eq:10}, the difference of the right sides of\eqref{sei} and \eqref{eq:4} equals
\begin{align}
 8(xy - xz +y z) (x^{2} + y^{2} + z^{2}) = 0.
\end{align}
Parameterizations appearing in Theorem \ref{fina} are
distinguished from equivalent representations in
\cite{MR1600120,MR2511669,liusom,ramlost} by the apparent
coefficient symmetry. These representations are analogous to balanced quintic parameterizations from \cite{qeis}.   %We
%conjecture that for general weight Eisenstein series on the full modular group and
%those for $\Gamma_{0}(7)$, the symmetric representations are uniquely determined by requiring the coefficient of each of the monomials $x^{m} y^{n} z^{\ell}$ have the same absolute value for each combination of $m, n$ and $\ell$ with $m+n+\ell=$ degree.
Symmetric septic representations for more general Eisenstein
series will be explored in a subsequent paper.

%\bibliographystyle{plain}
%\bibliography{timsbib}

\end{document}